\documentclass[a4paper]{amsart}
\usepackage{amsmath}
\usepackage{amsfonts}

\newtheorem{theorem}{Theorem}
\theoremstyle{plain}
\newtheorem{acknowledgement}{Acknowledgement}

\newtheorem{example}{Example}

\newtheorem{lemma}{Lemma}

\input{tcilatex}

\begin{document}
\title[Inverse Sturm-Liouville problem ]{Inverse nodal problem for Dirac
operator with integral type nonlocal boundary conditions}
\author{A. Sinan Ozkan}
\curraddr{Department of Mathematics, Faculty of Science, Sivas Cumhuriyet
University 58140 Sivas, TURKEY}
\email{sozkan@cumhuriyet.edu.tr}
\author{\.{I}brahim Adalar}
\curraddr{Zara Veysel Dursun Colleges of Applied Sciences, Sivas Cumhuriyet
University Zara/Sivas, TURKEY}
\email{iadalar@cumhuriyet.edu.tr}

\begin{abstract}
In this paper, Dirac operator with some integral type nonlocal boundary
conditions is studied. We show that the coefficients of the problem can be
uniquely determined by a dense set of nodal points. Moreover, we give an
algorithm for the reconstruction of some coefficients of the operator.
\end{abstract}

\thanks{The author thanks to the reviewers for constructive comments and
recommendations which help to improve the readability and quality of the
paper.}
\maketitle

\section{\textbf{Introduction}}

The inverse nodal problem, posed and solved firstly by McLaughlin for a
Sturm-Liouville operator \cite{mc1}. McLaughlin showed that the potential of
a Sturm-Liouville problem can be determined by a given dense subset of nodal
points. Later on, Hald and McLaughlin give some numerical schemes for the
reconstruction of the potential \cite{H}. X.F. Yang gave an algorithm for
the solution of the inverse nodal Sturm-Liouville problem \cite{yang}. His
method is the first of the algorithms still in use today. Inverse nodal
problems for various Sturm-Liouville operators have been studied in\ the
papers (\cite{yangx}-\cite{guo2} ).

Nonlocal boundary conditions appear when we cannot measure data directly at
the boundary. This kind conditions arise in various some applied problems of
biology, biotechnology, physics and etc. Two types of nonlocal boundary
conditions come to the fore. One class of them is called integral type
conditions, and the other is the Bitsadze-Samarskii-type conditions.

Some inverse problems for a class of Sturm-Liouville operators with nonlocal
boundary conditions are investigated in \cite{niz,niz2}. In particular,
inverse nodal problems for this-type operators with different nonlocal
integral boundary conditions are studied in (\cite{yang6}-\cite{yan6}).

The inverse nodal problems for Dirac operators with local and separated
boundary conditions are studied in (\cite{jde}-\cite{yang10}). In their
works, it is shown that the zeros of the first components of the
eigenfunctions determines the coefficients of operator. In \cite{yangyurko},
nonlocal conditions together with the Dirac system are considered. They give
some uniqueness theorems according to the classical spectral data. As far as
we know, inverse nodal problem for the Dirac system with nonlocal boundary
conditions has not been considered before.

In the present paper, we consider Dirac operator under some integral type
nonlocal boundary conditions and obtain the uniqueness of coefficients of
the problem according to a set of nodal points. Moreover, we give an
algorithm for the reconstruction of these coefficients.

We consider the boundary value problem $L$ generated by the system of Dirac
differential equations system:%
\begin{equation}
\ell \left[ Y(x)\right] :=BY^{\prime }(x)+\Omega (x)Y(x)=\lambda Y(x),\text{
\ }x\in (0,\pi ),  \label{1}
\end{equation}%
subject to the boundary conditions

\begin{eqnarray}
U(y) &:&=y_{1}(0)\sin \alpha +y_{2}(0)\cos \alpha -I_{f}(Y)=0\bigskip
\label{2} \\
V(y) &:&=y_{1}(\pi )\sin \beta +y_{2}(\pi )\cos \beta -I_{g}(Y)=0\bigskip
\label{3}
\end{eqnarray}%
where $B=\left( 
\begin{array}{cc}
0 & 1 \\ 
-1 & 0%
\end{array}%
\right) ,$ $\Omega (x)=\left( 
\begin{array}{cc}
V(x)+m & 0 \\ 
0 & V(x)-m%
\end{array}%
\right) ,$ $Y(x)=\left( 
\begin{array}{c}
y_{1}(x) \\ 
y_{2}(x)%
\end{array}%
\right) \bigskip $, $I_{f}(Y)=\int_{0}^{\pi }\left(
f_{1}(x)y_{1}(x)+f_{2}(x)y_{2}(x)\right) dx,$ $I_{g}(Y)=\int_{0}^{\pi
}\left( g_{1}(x)y_{1}(x)+g_{2}(x)y_{2}(x)\right) dx$,$\bigskip $ $\alpha $, $%
\beta $ and $m$ are real constants and $\lambda $ is the spectral parameter.
We assume $\Omega (x)$, $f_{i}(x)$ and $g_{i}(x)$ are real-valued functions
in the class of $W_{2}^{1}(0,\pi )$ for $i=1,2$ and put $F_{0}:=f_{1}\left(
0\right) \sin \alpha +f_{2}(0)\cos \alpha ,$ $F_{\pi }:=f_{1}\left( \pi
\right) \sin \beta +f_{2}(\pi )\cos \beta ,$ $G_{0}:=g_{1}\left( 0\right)
\sin \alpha +g_{2}(0)\cos \alpha $ and $G_{\pi }:=g_{1}\left( \pi \right)
\sin \beta +g_{2}(\pi )\cos \beta $.

\section{\textbf{Spectral properties of the problem}}

Let $S(x,\lambda )=\left( 
\begin{array}{c}
S_{1}(x,\lambda ) \\ 
S_{2}(x,\lambda )%
\end{array}%
\right) $ and $C(x,\lambda )=\left( 
\begin{array}{c}
C_{1}(x,\lambda ) \\ 
C_{2}(x,\lambda )%
\end{array}%
\right) $ be the solutions of (\ref{1}) under the initial conditions 
\begin{eqnarray*}
S_{1}(0,\lambda ) &=&0\text{, }S_{2}(0,\lambda )=-1\medskip \\
C_{1}(0,\lambda ) &=&1\text{, }C_{2}(0,\lambda )=0\medskip
\end{eqnarray*}%
respectively. $C(x,\lambda )$ and $S(x,\lambda )$ are entire functions of $%
\lambda $ for each fixed $x$ and the following asymptotic relations hold for 
$\left\vert \lambda \right\vert \rightarrow \infty $ (see \cite{jde}).

\begin{equation}
C_{1}(x,\lambda )=\cos (\lambda x-\omega (x))+\dfrac{m^{2}x}{2\lambda }\sin
(\lambda x-\omega (x))+o\left( \dfrac{e^{\left\vert \tau \right\vert x}}{%
\lambda }\right) \bigskip ,\medskip  \label{4}
\end{equation}%
\begin{equation}
C_{2}(x,\lambda )=\sin (\lambda x-\omega (x))-\dfrac{m}{\lambda }\sin
(\lambda x-\omega (x))-\dfrac{m^{2}x}{2\lambda }\cos (\lambda x-\omega
(x))+o\left( \dfrac{e^{\left\vert \tau \right\vert x}}{\lambda }\right) ,
\label{5}
\end{equation}%
\begin{equation}
S_{1}(x,\lambda )=\sin (\lambda x-\omega (x))+\dfrac{m}{\lambda }\sin
(\lambda x-\omega (x))-\dfrac{m^{2}x}{2\lambda }\cos (\lambda x-\omega
(x))+o\left( \dfrac{e^{\left\vert \tau \right\vert x}}{\lambda }\right)
,\medskip  \label{6}
\end{equation}%
\begin{equation}
S_{2}(x,\lambda )=-\cos (\lambda x-\omega (x))-\dfrac{m^{2}x}{2\lambda }\sin
(\lambda x-\omega (x))+o\left( \dfrac{e^{\left\vert \tau \right\vert x}}{%
\lambda }\right) ,  \label{7}
\end{equation}%
where $\omega (x)=\int_{0}^{x}V(t)dt$ and $\tau =$Im$\lambda .$

The characteristic function of problem (\ref{1})-(\ref{3}) is%
\begin{equation}
\Delta (\lambda )=\det \left( 
\begin{array}{cc}
U(C) & U(S)\medskip \\ 
V(C) & V(S)\medskip%
\end{array}%
\right)  \label{8}
\end{equation}%
and the zeros of $\Delta (\lambda )$ coincide with the eigenvalues of the
problem (\ref{1})-(\ref{3}). Clearly, $\Delta (\lambda )$ is entire function
and so the problem has a discrete spectrum. \ 

Using the asymptotic formulas (\ref{4})-(\ref{7}) in (\ref{8}), one can
easily obtain

\begin{eqnarray*}
\Delta (\lambda ) &=&\sin (\lambda \pi -\omega (\pi )+\beta -\alpha )-\dfrac{%
m^{2}\pi }{2\lambda }\cos (\lambda \pi -\omega (\pi )+\beta -\alpha )\bigskip
\\
&&-\dfrac{1}{\lambda }\left( f_{1}\left( \pi \right) \sin \beta +f_{2}\left(
\pi \right) \cos \beta -f_{2}\left( 0\right) \cos (\lambda \pi -\omega (\pi
)+\beta )\right) \bigskip \\
&&-\dfrac{1}{\lambda }\left( g_{1}\left( \pi \right) \sin (\lambda \pi
-\omega (\pi )-\alpha )-g_{2}\left( \pi \right) \cos (\lambda \pi -\omega
(\pi )-\alpha )\right) \bigskip \\
&&+\dfrac{1}{\lambda }\left( f_{1}\left( 0\right) \sin (\lambda \pi -\omega
(\pi )+\beta )-g_{1}\left( 0\right) \sin \alpha -g_{2}\left( 0\right) \cos
\alpha \right) \bigskip \\
&&-\dfrac{m}{\lambda }\sin (\lambda \pi -\omega (\pi )+\beta -\alpha )\cos
(\beta +\alpha )\cos (\beta -\alpha )\bigskip \\
&&+\dfrac{m}{\lambda }\cos (\lambda \pi -\omega (\pi )+\beta -\alpha )\cos
(\beta +\alpha )\sin (\beta -\alpha )+o\left( \dfrac{e^{\left\vert \tau
\right\vert \pi }}{\lambda }\right)
\end{eqnarray*}%
for sufficiently large $\left\vert \lambda \right\vert .$ Let $\left\{
\lambda _{n}:n=0,\pm 1,\pm 2,...\right\} _{n\geq 0}$ be the set of
eigenvalues. Since $\Delta (\lambda )=\sin (\lambda \pi -\omega (\pi )+\beta
-\alpha )+O(\dfrac{e^{\left\vert \tau \right\vert x}}{\lambda })$, the
numbers $\lambda _{n}$ are real for $\left\vert n\right\vert \rightarrow
\infty $ and\ satisfy the following asimptotic formula%
\begin{equation*}
\lambda _{n}=n+\omega (\pi )-\beta +\alpha +O(\dfrac{1}{n}),\text{ \ \ }%
\left\vert n\right\vert \rightarrow \infty
\end{equation*}%
Moreover, it is clear that we can write the following equation. 
\begin{eqnarray*}
&&\bigskip \\
\left( 1+O\left( \frac{1}{\lambda _{n}}\right) \right) \tan (\lambda _{n}\pi
-\omega (\pi )+\beta -\alpha ) &=&\dfrac{1}{\lambda _{n}}\left[ \dfrac{%
m^{2}\pi }{2}-F_{0}-G_{\pi }+m\cos (\beta +\alpha )\sin (\alpha -\beta )%
\right] \\
&&+\dfrac{F_{\pi }+G_{0}}{\lambda _{n}\cos (\lambda \pi -\omega (\pi )+\beta
-\alpha )}+o\left( \dfrac{1}{\lambda _{n}}\right) .
\end{eqnarray*}%
\newline
Since $\left( 1+O(\frac{1}{\lambda _{n}})\right) ^{-1}=1+O(\frac{1}{\lambda
_{n}}),$\ \ this implies that%
\begin{equation}
\tan (\lambda _{n}\pi -\omega (\pi )+\beta -\alpha )=\dfrac{A_{1}}{\lambda
_{n}}+\frac{A_{2}}{\lambda _{n}\cos (\lambda \pi -\omega (\pi )+\beta
-\alpha )}+o\left( \dfrac{1}{\lambda _{n}}\right) \newline
\label{9}
\end{equation}%
\bigskip for sufficiently large $n$ where%
\begin{equation*}
A_{1}=\dfrac{m^{2}\pi }{2}+m\cos (\beta +\alpha )\sin (\alpha -\beta
)-\left( F_{0}+G_{\pi }\right) ,
\end{equation*}%
\begin{equation*}
A_{2}=F_{\pi }+G_{0}.
\end{equation*}%
From (\ref{9}), we can see that 
\begin{equation*}
\lambda _{n}\pi -\omega (\pi )+\beta -\alpha =n\pi +O(\frac{1}{n}),\text{\ \ 
}\left\vert n\right\vert \rightarrow \infty .
\end{equation*}%
Hence 
\begin{eqnarray}
\frac{1}{\cos (\lambda _{n}\pi -\omega (\pi )+\beta -\alpha )} &=&\frac{1}{%
\cos (n\pi +O(\frac{1}{n}))}\bigskip  \notag \\
&=&\left( -1\right) ^{n}(1+o(\frac{1}{n})),\text{\ \ }\left\vert
n\right\vert \rightarrow \infty .\bigskip  \label{10}
\end{eqnarray}%
Using (\ref{9}) and (\ref{10}) together, we obtain the following lemma.

\begin{lemma}
The following asymptotic relation is valid for $\left\vert n\right\vert
\rightarrow \infty .$%
\begin{equation}
\lambda _{n}=n+\dfrac{1}{\pi }\int\limits_{0}^{\pi }V(t)dt+\dfrac{\alpha
-\beta }{\pi }+\dfrac{A_{1}}{n\pi }+\frac{\left( -1\right) ^{n}A_{2}}{n\pi }%
+o\left( \dfrac{1}{n}\right) .  \label{11}
\end{equation}
\end{lemma}

\section{\textbf{Nodal points}}

Let $\varphi (x,\lambda _{n})=\left( 
\begin{array}{c}
\varphi _{1}(x,\lambda _{n}) \\ 
\varphi _{2}(x,\lambda _{n})%
\end{array}%
\right) $ be the eigenfunction of (\ref{1})-(\ref{3}) corresponding to the
eigenvalue $\lambda _{n}.$ It is clear that%
\begin{equation}
\varphi _{i}(x,\lambda _{n})=U(S(x,\lambda _{n}))C_{i}(x,\lambda
_{n})-U(C(x,\lambda _{n}))S_{i}(x,\lambda _{n}),\text{ }i=1,2.  \label{a}
\end{equation}

From (\ref{a}) and Lemma 1, we can give the following asymptotic formula for
sufficiently large $\left\vert n\right\vert $%
\begin{eqnarray*}
\varphi _{1}(x,\lambda _{n}) &=&\left. -\cos (\lambda _{n}x-\omega
(x)-\alpha )-\dfrac{m^{2}x}{2\lambda _{n}}\sin (\lambda _{n}x-\omega
(x)-\alpha )\right. \bigskip \\
&&\left. -\frac{m\sin 2\alpha }{2\lambda _{n}}\sin (\lambda _{n}x-\omega
(x)-\alpha )-\frac{m\sin ^{2}\alpha }{\lambda _{n}}\cos (\lambda
_{n}x-\omega (x)-\alpha )\right. \bigskip \\
&&\left. +\frac{\left( -1\right) ^{n}f_{1}\left( \pi \right) }{\lambda _{n}}%
(\cos \beta \cos (\lambda _{n}x-\omega (x)-\alpha )-\sin \beta \sin (\lambda
_{n}x-\omega (x)-\alpha ))\right. \bigskip \\
&&\left. -\frac{\left( -1\right) ^{n}f_{2}\left( \pi \right) }{\lambda _{n}}%
(\cos \beta \sin (\lambda _{n}x-\omega (x)-\alpha )+\sin \beta \cos (\lambda
_{n}x-\omega (x)-\alpha ))\right. \bigskip \\
&&\left. +\frac{f_{1}\left( 0\right) }{\lambda _{n}}(\sin \alpha \sin
(\lambda _{n}x-\omega (x)-\alpha )-\cos \alpha \cos (\lambda _{n}x-\omega
(x)-\alpha ))\right. \bigskip \\
&&\left. +\frac{f_{2}\left( 0\right) }{\lambda _{n}}(\cos \alpha \sin
(\lambda _{n}x-\omega (x)-\alpha )+\sin \alpha \cos (\lambda _{n}x-\omega
(x)-\alpha ))\right. \bigskip \\
&&\left. +o\left( \dfrac{e^{\left\vert \tau \right\vert x}}{\lambda }\right)
.\right.
\end{eqnarray*}

\begin{lemma}
For sufficiently large positive $n$, $\varphi _{1}(x,\lambda _{n})$ has
exactly $n$ zeros, namely nodal points, $\left\{ x_{n}^{j}:j=\overline{0,n-1}%
\right\} $ in the interval $\left( 0,\pi \right) $:\newline
The numbers $\left\{ x_{n}^{j}\right\} $ satisfy the following asymptotic
formula:%
\begin{equation}
\begin{array}{l}
x_{n}^{j}=\dfrac{\left( j+1/2\right) \pi }{n}+\dfrac{\omega
(x_{n}^{j})+\alpha }{n}-\dfrac{\left( j+1/2\right) }{n}\left( \dfrac{\omega
(\pi )+\alpha -\beta }{n}\right) \bigskip \\ 
\text{ \ \ }-\dfrac{\omega (\pi )+\alpha -\beta }{n^{2}\pi }\left( \omega
(x_{n}^{j})+\alpha \right) +\dfrac{1}{n^{2}}\left( \frac{m^{2}x_{n}^{j}}{2}+%
\frac{m\sin 2\alpha }{2}+(-1)^{n}F_{\pi }-F_{0}\right) \bigskip \\ 
\text{ \ \ }-\dfrac{\left( j+1/2\right) }{n}\left( \dfrac{A_{1}+\left(
-1\right) ^{n}A_{2}}{n^{2}}\right) +o\left( \frac{1}{n^{2}}\right) .\bigskip%
\end{array}
\label{13}
\end{equation}
\end{lemma}

\begin{proof}
Let $\kappa _{n,j}:=\lambda _{n}x_{n}^{j}-\omega (x_{n}^{j})-\alpha $. From $%
\varphi _{1}(x_{n}^{j},\lambda _{n})=0,$ we can write for $n\rightarrow
\infty $%
\begin{eqnarray*}
0 &=&-\cos \kappa _{n,j}-\dfrac{m\sin 2\alpha }{2\lambda _{n}}\sin \kappa
_{n,j}-\dfrac{m\sin ^{2}\alpha }{\lambda _{n}}\cos \kappa _{n,j}-\dfrac{%
m^{2}x_{n}^{j}}{2\lambda _{n}}\sin \kappa _{n,j}+\bigskip \\
&&+\frac{\left( -1\right) ^{n}f_{1}\left( \pi \right) }{\lambda _{n}}\left(
\cos \kappa _{n,j}\cos \beta -\sin \kappa _{n,j}\sin \beta \right) +\bigskip
\\
&&-\frac{\left( -1\right) ^{n}f_{2}\left( \pi \right) }{\lambda _{n}}\left(
\sin \kappa _{n,j}\cos \beta +\cos \kappa _{n,j}\sin \beta \right) +\bigskip
\\
&&+\frac{f_{1}\left( 0\right) }{\lambda _{n}}\left( \sin \kappa _{n,j}\sin
\alpha -\cos \kappa _{n,j}\cos \alpha \right) +\bigskip \\
&&+\bigskip \frac{f_{2}\left( 0\right) }{\lambda _{n}}\left( \sin \kappa
_{n,j}\cos \alpha +\cos \kappa _{n,j}\sin \alpha \right) +o\left( \dfrac{%
e^{\left\vert \tau \right\vert \pi }}{\lambda _{n}}\right) .
\end{eqnarray*}%
$\bigskip $\newline

It can be obtained that%
\begin{eqnarray*}
\tan \left( \kappa _{n,j}-\dfrac{\pi }{2}\right) &=&\frac{\dfrac{%
m^{2}x_{n}^{j}}{2\lambda _{n}}+\dfrac{m\sin 2\alpha }{2\lambda _{n}}+\frac{%
\left( -1\right) ^{n}F_{\pi }-F_{0}}{\lambda _{n}}+o\left( \dfrac{1}{\lambda
_{n}}\right) }{1+O(\dfrac{1}{\lambda _{n}})} \\
&=&\dfrac{m^{2}x_{n}^{j}}{2\lambda _{n}}+\dfrac{m\sin 2\alpha }{2\lambda _{n}%
}+\frac{\left( -1\right) ^{n}F_{\pi }-F_{0}}{\lambda _{n}}+o\left( \dfrac{1}{%
\lambda _{n}}\right)
\end{eqnarray*}

Taking into account Taylor's expansion formulas of the arctangent we get$%
\medskip $ 
\begin{equation*}
\kappa _{n,j}=\left( j+\frac{1}{2}\right) \pi +\dfrac{1}{\lambda _{n}}\left( 
\frac{m^{2}x_{n}^{j}}{2}+\frac{m\sin 2\alpha }{2}+(-1)^{n}F_{\pi
}-F_{0}\right) +o\left( \dfrac{1}{\lambda _{n}}\right) .\medskip
\end{equation*}

It follows from the last equality

\begin{equation*}
x_{n}^{j}=\dfrac{\left( j+\frac{1}{2}\right) \pi +\omega (x_{n}^{j})+\alpha 
}{\lambda _{n}}+\dfrac{1}{\lambda _{n}^{2}}\left( \frac{m^{2}x_{n}^{j}}{2}+%
\frac{m\sin 2\alpha }{2}+(-1)^{n}F_{\pi }-F_{0}\right) +o\left( \dfrac{1}{%
\lambda _{n}^{2}}\right) .
\end{equation*}%
Finally, using the asymptotic formula%
\begin{equation*}
\lambda _{n}^{-1}=\dfrac{1}{n}\left\{ 1-\dfrac{\omega (\pi )+\alpha -\beta }{%
n\pi }-\dfrac{A_{1}}{n^{2}\pi }-\frac{\left( -1\right) ^{n}A_{2}}{n^{2}\pi }%
+o\left( \dfrac{1}{n^{3}}\right) \right\}
\end{equation*}%
we can obtain our desired formula: (\ref{13}).
\end{proof}

\section{\textbf{Inverse nodal problem}}

Let $X$ be the set of nodal points and $\omega (\pi )=0.$ For each fixed $%
x\in \left( 0,\pi \right) $ $\ $we can choose a sequence $\left(
x_{n}^{j(n)}\right) \subset X$ so that $x_{n}^{j(n)}$ converges to $x.$ Then
the following limits are exist and finite:%
\begin{equation}
\underset{\left\vert n\right\vert \rightarrow \infty }{\lim }n\left(
x_{n}^{j(n)}-\frac{\left( j(n)+\frac{1}{2}\right) \pi }{n}\right) =\psi
_{1}(x),  \label{14}
\end{equation}%
where%
\begin{equation*}
\psi _{1}(x)=\omega (x)+\frac{\left( \beta -\alpha \right) x}{\pi }+\alpha
\end{equation*}%
and%
\begin{equation}
\underset{\left\vert n\right\vert \rightarrow \infty }{\lim }n^{2}\pi \left(
x_{n}^{j(n)}-\frac{\left( j(n)+\frac{1}{2}\right) \pi +\left( \omega
(x_{n}^{j(n)})+\alpha \right) }{n}+\dfrac{\left( j(n)+1/2\right) }{n}\left( 
\dfrac{\alpha -\beta }{n}\right) \right) =\psi _{2}(x),  \label{15}
\end{equation}%
where%
\begin{equation*}
\psi _{2}(x)=\left\{ 
\begin{array}{cc}
\psi _{2}^{+}(x), & n\text{ is even}\medskip \\ 
\psi _{2}^{-}(x), & n\text{ is odd}%
\end{array}%
\right.
\end{equation*}%
and%
\begin{eqnarray*}
\psi _{2}^{+}(x) &=&\left( \beta -\alpha \right) \left( \omega (x)+\alpha
\right) -x\left( A_{1}+A_{2}\right) +\pi \left( \frac{m^{2}x}{2}+\frac{m\sin
2\alpha }{2}+F_{\pi }-F_{0}\right) ,\medskip \\
\psi _{2}^{-}(x) &=&\left( \beta -\alpha \right) \left( \omega (x)+\alpha
\right) -x\left( A_{1}-A_{2}\right) +\pi \left( \frac{m^{2}x}{2}+\frac{m\sin
2\alpha }{2}-F_{\pi }-F_{0}\right) .
\end{eqnarray*}%
Therefore, proof of the following theorem is clear.

\begin{theorem}
The given dense subset of nodal points $X$ uniquely determines the potential 
$V(x)$ a.e. on $\left( 0,\pi \right) $ and the coefficients $\alpha $ and $%
\beta $ in the boundary conditions; Moreover, $V(x),$ $\alpha $ and $\beta $
can be reconstructed by the following formulae:

\textbf{Step-1:} For each fixed $x\in (0,\pi ),$ choose a sequence $\left(
x_{n}^{j(n)}\right) \subset X$ such that $\underset{\left\vert n\right\vert
\rightarrow \infty }{\lim }x_{n}^{j(n)}=x;$

\textbf{Step-2: }Find the function $\psi _{1}(x)$ and $\psi _{2}(x)$ from (%
\ref{14})\ and (\ref{15}), and calculate 
\begin{eqnarray*}
\alpha &=&\psi _{1}(0)\medskip \\
\beta &=&\psi _{1}(\pi )\medskip \\
V(x) &=&\psi _{1}^{\prime }(x)+\frac{\alpha -\beta }{\pi }\medskip \\
F_{\pi } &=&\frac{\psi _{2}^{+}(0)-\psi _{2}^{-}(0)}{2\pi }\medskip \\
G_{0} &=&\frac{\psi _{2}^{-}(\pi )-\psi _{2}^{+}(\pi )}{2\pi }\medskip .
\end{eqnarray*}

Additionally, if $m$ is known, then $F_{0}$ and $G_{\pi }$ can be found by
the following formulas%
\begin{eqnarray*}
F_{0} &=&-\frac{\psi _{2}^{+}(0)+\psi _{2}^{-}(0)}{2}+\alpha \left( \beta
-\alpha \right) +m\pi \frac{\sin 2\alpha }{2},\bigskip \\
G_{\pi } &=&\frac{\psi _{2}^{-}(\pi )+\psi _{2}^{+}(\pi )}{2\pi }-\frac{%
\alpha \left( \beta -\alpha \right) }{\pi }-m\frac{\sin 2\beta }{2}.\bigskip
\end{eqnarray*}
\end{theorem}

\begin{example}
Let $\left\{ x_{n}^{j}\right\} \subset X$ be the dense subset of nodal
points in $(0,\pi )$ given by the following asimptotics:\newline
$x_{n}^{j}=\dfrac{\left( j+1/2\right) \pi }{n}+\dfrac{\sin \dfrac{\left(
j+1/2\right) \pi }{n}+\frac{\pi }{6}}{n}+\dfrac{\left( j+1/2\right) \pi }{%
6n^{2}}$

$+\dfrac{1}{6n^{2}}\left( \sin \dfrac{\left( j+1/2\right) \pi }{n}+\frac{\pi 
}{6}\right) +\dfrac{1}{n^{2}}\left( \frac{3}{2}\left( \dfrac{\left(
j+1/2\right) \pi }{n}+\frac{1}{2}\right) +(-1)^{n}2\pi \right) $

$-\dfrac{\left( j+1/2\right) }{n}\left( \dfrac{\left( 2-\sqrt{3}\right) \pi
+\left( -1\right) ^{n}4\pi }{2n^{2}}\right) +o\left( \frac{1}{n^{2}}\right) $%
. $\bigskip $\newline
It can \ be calculated from (\ref{14})\ and (\ref{15}) that,%
\begin{eqnarray*}
\psi _{1}(x) &=&\sin x+\frac{1}{6}\left( x+\pi \right) \\
\psi _{2}(x) &=&\left\{ 
\begin{array}{cc}
\frac{\pi }{6}\left( \sin x+\frac{\pi }{6}\right) -\frac{\pi }{2}\left( 6-%
\sqrt{3}\right) x+\pi \left( \frac{3}{2}\left( x+\frac{1}{2}\right) +2\pi
\right) , & n\text{ is even}\medskip \\ 
\frac{\pi }{6}\left( \sin x+\frac{\pi }{6}\right) +\frac{\pi }{2}\left( 2+%
\sqrt{3}\right) x+\pi \left( \frac{3}{2}\left( x+\frac{1}{2}\right) -2\pi
\right) , & n\text{ is odd}%
\end{array}%
\right.
\end{eqnarray*}%
Therefore, it is obtained by using the algorithm in Theorem 1,%
\begin{eqnarray*}
&&\left. \alpha =\frac{\pi }{6},\medskip \right. \\
&&\left. \beta =\frac{\pi }{3},\medskip \right. \\
&&\medskip \left. V(x)=\cos x,\right. \\
&&\left. F_{\pi }=2\pi ,\right. \\
&&\left. G_{0}=0.\right.
\end{eqnarray*}%
If $m=\sqrt{3}$, then%
\begin{eqnarray*}
&&\left. F_{0}=0,\bigskip \right. \\
&&\left. G_{\pi }=\frac{\left( \sqrt{3}+1\right) \pi }{2}.\bigskip \bigskip
\right. \bigskip
\end{eqnarray*}
\end{example}

\begin{acknowledgement}
The authors would like to thank the referees for their valuable comments
which helped to improve the manuscript.
\end{acknowledgement}


\begin{thebibliography}{99}
\bibitem{mc1} McLaughlin, J. R. (1988). Inverse spectral theory using nodal
points as data---a uniqueness result. Journal of Differential Equations,
73(2), 354-362.

\bibitem{H} Hald, O. H., \& McLaughlin, J. R. (1989). Solution of inverse
nodal problems. Inverse problems, 5(3), 307-347.

\bibitem{yang} Yang, X. F. (1997). A solution of the inverse nodal problem.
Inverse Problems, 13(1), 203-213.

\bibitem{yangx} Yang, X. F. (2001). A new inverse nodal problem. Journal of
Differential Equations, 169(2), 633-653.

\bibitem{but} Buterin, S. A., \& Shieh, C. T. (2009). Inverse nodal problem
for differential pencils. Applied Mathematics Letters, 22(8), 1240-1247.

\bibitem{but2} Buterin, S. A., \& Shieh, C. T. (2012). Incomplete inverse
spectral and nodal problems for differential pencils. Results in
Mathematics, 62(1), 167-179.

\bibitem{ch} Cheng, Y. H., Law, C. K., \& Tsay, J. (2000). Remarks on a new
inverse nodal problem. Journal of Mathematical Analysis and Applications,
248(1), 145-155.

\bibitem{Law} Law, C. K., \& Tsay, J. (2001). On the well-posedness of the
inverse nodal problem. Inverse problems, 17(5), 1493-1512.

\bibitem{law2} Law, C. K., Shen, C. L., \& Yang, C. F. (1999). The inverse
nodal problem on the smoothness of the potential function. Inverse Problems,
15(1), 253.

\bibitem{law3} Law, C. K., \& Yang, C. F. (1998). Reconstructing the
potential function and its derivatives using nodal data. Inverse Problems,
14(2), 299--312.

\bibitem{yang1} Yang, C. F. (2013). Inverse nodal problems of discontinuous
Sturm--Liouville operator. Journal of Differential Equations, 254(4),
1992-2014.

\bibitem{yang2} Yang, C. F., \& Yang, X. P. (2011). Inverse nodal problems
for the Sturm-Liouville equation with polynomially dependent on the
eigenparameter. Inverse Problems in Science and Engineering, 19(7), 951-961.

\bibitem{yang3} Yang, C. F. (2012). Inverse nodal problems for the
Sturm-Liouville operator with eigenparameter dependent boundary conditions.
Operators and Matrices, 6(1), 63-77.

\bibitem{ASO} Ozkan, A. S., \& Keskin, B. (2015). Inverse nodal problems for
Sturm--Liouville equation with eigenparameter-dependent boundary and jump
conditions. Inverse Problems in Science and Engineering, 23(8), 1306-1312.

\bibitem{yur1} Shieh, C. T., \& Yurko, V. A. (2008). Inverse nodal and
inverse spectral problems for discontinuous boundary value problems. Journal
of Mathematical Analysis and Applications, 347(1), 266-272.

\bibitem{yur2} Wang, Y. P., \& Yurko, V. A. (2016). On the inverse nodal
problems for discontinuous Sturm--Liouville operators. Journal of
Differential Equations, 260(5), 4086-4109.

\bibitem{yur3} Yurko V. A., (2008) Inverse nodal problems for
Sturm--Liouville operators on star-type graphs, J. Inverse Ill-Posed Probl.
16 715--722.

\bibitem{yur4} Yurko V. A., Inverse Spectral Problems for Differential
Operators and Their Applications, Gordon and Breach, Amsterdam, 2000.

\bibitem{Shieh} Shen, C. L., \& Shieh, C. T. (2000). An inverse nodal
problem for vectorial Sturm-Liouville equations. Inverse Problems, 16(2),
349--356.

\bibitem{bond} Hu, Y. T., Bondarenko, N. P., \& Yang, C. F. (2020). Traces
and inverse nodal problem for Sturm--Liouville operators with frozen
argument. Applied Mathematics Letters, 102, 106096.

\bibitem{bond2} Hu, Y. T., Bondarenko, N. P., Shieh, C. T., \& Yang, C. F.
(2019). Traces and inverse nodal problems for Dirac-type
integro-differential operators on a graph. Applied Mathematics and
Computation, 363, 124606.

\bibitem{koy1} Koyunbakan, H. (2008). The inverse nodal problem for a
differential operator with an eigenvalue in the boundary condition. Applied
mathematics letters, 21(12), 1301-1305.

\bibitem{guo2} Guo, Y., \& Wei, G. (2013). Inverse problems: dense nodal
subset on an interior subinterval. Journal of Differential Equations,
255(7), 2002-2017.

\bibitem{niz} Albeverio, S., Hryniv, R. O., \& Nizhnik, L. P. (2007).
Inverse spectral problems for non-local Sturm--Liouville operators. Inverse
problems, 23(2), 523.

\bibitem{niz2} Nizhnik, L. (2010). Inverse nonlocal Sturm--Liouville
problem. Inverse problems, 26(12), 125006.

\bibitem{yang6} Xu, X. J., \& Yang, C. F. (2019). Inverse nodal problem for
nonlocal differential operators. Tamkang Journal of Mathematics, 50(3),
337-347.

\bibitem{yang4} Hu, Y. T., Yang, C. F., \& Xu, X. C. (2017). Inverse nodal
problems for the Sturm--Liouville operator with nonlocal integral
conditions. Journal of Inverse and Ill-Posed Problems, 25(6), 799-806.

\bibitem{yang5} Qin, X., Gao, Y., \& Yang, C. (2019). Inverse Nodal Problems
for the Sturm-Liouville Operator with Some Nonlocal Integral Conditions.
Journal of Applied Mathematics and Physics, 7(01), 111.

\bibitem{yan6} Yang, C. F. (2010). Inverse nodal problem for a class of
nonlocal sturm-liouville operator. Mathematical Modelling and Analysis,
15(3), 383-392.

\bibitem{jde} Keskin, B., \& Ozkan, A. S. (2017). Inverse nodal problems for
Dirac-type integro-differential operators. Journal of Differential
Equations, 263(12), 8838-8847.

\bibitem{bon} Bondarenko, N., \& Buterin, S. (2017). On recovering the Dirac
operator with an integral delay from the spectrum. Results in Mathematics,
71(3), 1521-1529.

\bibitem{koy6} Gulsen, T., Yilmaz, E., \& Koyunbakan, H. (2017). Inverse
nodal problem for p--laplacian dirac system. Mathematical Methods in the
Applied Sciences, 40(7), 2329-2335.

\bibitem{guoo} Guo, Y., \& Wei, G. (2015). Inverse nodal problem for Dirac
equations with boundary conditions polynomially dependent on the spectral
parameter. Results in Mathematics, 67(1), 95-110.

\bibitem{wang} Wang, Y. P., Yurko, V. A., \& Shieh, C. T. (2019). The
uniqueness in inverse problems for Dirac operators with the interior
twin-dense nodal subset. Journal of Mathematical Analysis and Applications,
479(1), 1383-1402.

\bibitem{pivo} Yang, C. F., \& Pivovarchik, V. N. (2013). Inverse nodal
problem for Dirac system with spectral parameter in boundary conditions.
Complex Analysis and Operator Theory, 7(4), 1211-1230.

\bibitem{yang10} Yang, C. F., \& Huang, Z. Y. (2010). Reconstruction of the
Dirac operator from nodal data. Integral Equations and Operator Theory,
66(4), 539-551.

\bibitem{yangyurko} Yang, C. F., \& Yurko, V. (2016). Recovering Dirac
operator with nonlocal boundary conditions. Journal of Mathematical Analysis
and Applications, 440(1), 155-166.
\end{thebibliography}
\end{document}